\documentclass{article}

\usepackage{graphicx}
\usepackage{indentfirst}
\usepackage{amsmath,amsfonts,amsthm,amssymb}
\usepackage{mathrsfs}
\usepackage{amscd}

\def\co{\colon\thinspace}
\DeclareMathAlphabet{\mathsfsl}{OT1}{cmss}{m}{sl}
\newcommand{\tensor}[1]{\mathsfsl{#1}}

\newtheorem{thm}{Theorem}[section]
\newtheorem{lem}[thm]{Lemma}
\newtheorem{cor}[thm]{Corollary}
\newtheorem{prop}[thm]{Proposition}
\newtheorem*{thm*}{Theorem}

\theoremstyle{definition}
\newtheorem{defn}[thm]{Definition}
\newtheorem{rem}[thm]{Remark}
\newtheorem{construction}[thm]{Construction}

\begin{document}

\title{Heegaard Floer homology and fibred $3$--manifolds}

\author{{Yi NI}\\{\normalsize Department of Mathematics, MIT, Room 2-306}\\
{\normalsize 77 Massachusetts Avenue, Cambridge, MA
02139-4307}\\{\small\it Emai\/l\/:\quad\rm yni@math.mit.edu}}

\date{}
\maketitle

\begin{abstract}
Given a closed $3$--manifold $Y$, we show that the Heegaard Floer
homology determines whether $Y$ fibres over the circle with a
fibre of negative Euler characteristic. This is an analogue of an
earlier result about knots proved by Ghiggini and the author.
\end{abstract}

{\it\hfill Dedicated to the memory of Xiao-Song Lin}

\section{Introduction}

Heegaard Floer homology was introduced by Ozsv\'ath and Szab\'o in
\cite{OSzAnn1}. This theory contains a package of invariants for
closed 3--manifolds. A filtered version of these invariants,
called knot Floer homology, was defined by Ozsv\'ath--Szab\'o
\cite{OSzKnot} and Rasmussen \cite{Ra2}
 for null-homologous knots. This theory turns out to be very powerful. For
example, it detects the Thurston norm \cite{OSzGenus}, and a
result due to Ghiggini \cite{Gh} and the author \cite{Ni2} states
that knot Floer homology detects fibred knots.

In fact, given a compact manifold with boundary, the information
from knot Floer homology tells you whether this manifold is
fibred. (See the proof of \cite[Corollary 1.2]{Ni2}.) Now it is
natural to ask whether a similar result can be proved for closed
manifolds. In the current paper, we will answer this question
affirmatively. Our main theorem is:

\begin{thm}\label{ClosedFibre}
Suppose $Y$ is a closed irreducible $3$--manifold, $F\subset Y$ is
a closed connected surface of genus $g\ge2$.  Let
$HF^+(Y,[F],g-1)$ denote the group
$$\bigoplus_{\mathfrak s\in \mathrm{Spin}^c(Y),\;\langle c_1(\mathfrak s),[F]\rangle=2g-2}HF^+(Y,\mathfrak s).$$
If $HF^+(Y,[F],g-1)\cong\mathbb Z$, then $Y$ fibres over the
circle with $F$ as a fibre.
\end{thm}

The converse of this theorem is already known, see \cite[Theorem
5.2]{OSzSympl}.

\begin{rem}
When $g=1$, there is no chance for $HF^+(Y,[F],0)$ to be $\mathbb
Z$; some arguments in this paper also break down. Nevertheless, it
is reasonable to expect that Heegaard Floer homology still detects
fibrations in the case of $g=1$. As suggested by Ozsv\'ath and
Szab\'o, in this case one may use some variants of Heegaard Floer
homology, for instance, Heegaard Floer homology with twisted
coefficients \cite[Section 8]{OSzAnn2} or with coefficients in the
Novikov ring \cite[11.0.1]{OSzAnn1}, \cite{JM}.
\end{rem}

This paper is closely related to the previous works of Ghiggini
\cite{Gh} and of the author \cite{Ni2}. We briefly recall the main
ingredients in the proof of the case of knots from those two
papers. There are three main ingredients:

\noindent(I)\quad The construction of two different taut
foliations under certain conditions. This part is contained in
\cite[Section 6]{Ni2} and the proof uses an argument due to Gabai
\cite{G3}. This part was also obtained by Ian Agol by a different
method, and the genus 1 case was proved in \cite{Gh}.

\noindent(II)\quad The existence of a taut foliation implies the
nontriviality of the corresponding Ozsv\'ath--Szab\'o contact
invariant. This part was almost proved by Ozsv\'ath and Szab\'o
\cite{OSzGenus}. But there are some technical restrictions in the
cases established by Ozsv\'ath and Szab\'o. For example, one needs
the condition $b_1(Y)=1$ or the use of twisted coefficients. The
case that is used in \cite{Gh,Ni2} was proved by Ghiggini in
\cite{Gh}. Basically, he proved a result in contact topology,
which, combined with some results of Ozsv\'ath and Szab\'o,
implies the desired nontriviality theorem.

\noindent(III)\quad Two decomposition formulas for knot Floer
homology, one in the case of horizontal decomposition
\cite[Theorem 4.1]{Ni2}, the other in the case of decomposition
along a separating product annulus \cite[Theorem 5.1]{Ni2}. The
second formula is essential and more technical. The proof of this
part uses the techniques introduced in \cite{Ni1}.

Now if (I) can be done, then (II) implies that the two distinct
taut foliations give rise to two linearly independent contact
invariants in the topmost term of the Heegaard Floer homology, so
the Floer homology is not monic. If (I) cannot be done, then the
topology of the knot complement is very restricted. One can then
use (III) to reduce the general case to the known case.

As we will find in this paper, the above ingredients (I) and (II)
can be applied to the case of closed manifolds without essential
changes. However, we are not able to prove an analogue of
\cite[Theorem~5.1]{Ni2} for closed manifolds. Instead, we will
construct a knot $K$ in a new manifold $Z$, and show that the pair
$(Z,K)$ has monic knot Floer homology by using twisted
coefficients and a simple argument in homological algebra. Then we
can apply \cite[Theorem 1.1]{Ni2} to get our conclusion.

\begin{rem}
Juh\'asz \cite{Ju} proved a very general decomposition formula for
the Floer homology of balanced sutured manifolds, based on the
techniques introduced by Sarkar and Wang \cite{SW}. In the case of
knots, the above ingredients (II) and (III) can be deduced from
this formula. This approach avoids the use of contact and
symplectic topology, but it is not clear to the author how to use
it to study closed manifolds.
\end{rem}

\noindent{\bf Acknowledgements.} The author wishes to thank Paolo
Ghiggini, whose work \cite{Gh} has a great influence on this one.
He also wishes to thank David Gabai, Feng Luo, Peter Ozsv\'ath,
Jacob Rasmussen and Zolt\'an Szab\'o for their encouragements
during the course of this work. Special thanks are due to Yinghua
Ai for a conversation during which the author realized a gap in an
earlier version of this paper, and for helpful discussions on
twisted Heegaard Floer homology. He is also grateful to the
referee for pointing out a mistake and for providing many
constructive suggestions.

The earlier versions of this paper were written when the author
was at Princeton University and Columbia University. The author
was partially supported by a Graduate School Centennial Fellowship
at Princeton University, an AIM Five-Year Fellowship and an NSF
grant (No. DMS-0805807) during the course of this work.

\section{Twisted Heegaard Floer homology}

For technical reason, at some point of the proof of our main
theorem we will use twisted Heegaard Floer homology with
coefficients in a Novikov ring. In this section, we will collect
some basic materials on this version of twisted Heegaard Floer
homology. More details about twisted Heegaard Floer homology can
be found in \cite{OSzAnn2,JM,AiP}.

\subsection{Twisted chain complexes}

Let $Y$ be a closed, oriented $3$--manifold.
$(\Sigma,\mbox{\boldmath$\alpha$},\mbox{\boldmath$\beta$},z)$ is a
Heegaard diagram for $Y$. We always assume the diagram satisfies a
certain admissibility condition so that the Heegaard Floer
invariants we are considering are well-defined (see \cite{OSzAnn1}
for more details).

Let
$$\mathcal L=\mathbb Q[T^{-1},T]]=\left\{\sum_{i=n}^{+\infty}a_iT^i\bigg|\:a_i\in\mathbb Q,\;n\in\mathbb Z\right\}$$
be a Novikov ring, which is actually a field.

Let $\omega\subset Y$ be an immersed, possibly disconnected,
closed, oriented curve. One can homotope $\omega$ to be a curve
$\omega'\subset\Sigma$, such that $\omega'$ is in general position
with the $\alpha$-- and $\beta$--curves, namely, $\omega'$ is
transverse to these curves, and $\omega'$ does not contain any
intersection point of $\alpha$-- and $\beta$--curves.

Let $\underline{CF}^{\infty}(Y,\omega;\mathcal L)$ be the
$\mathcal L$--module freely generated by $[\mathbf x,i]$, where
$\mathbf x\in\mathbb T_{\alpha}\cap\mathbb T_{\beta}$,
$i\in\mathbb Z$. If $\phi$ is a topological Whitney disk
connecting $\mathbf x$ to $\mathbf y$, let
$\partial_{\alpha}\phi=(\partial\phi)\cap\mathbb T_{\alpha}$. We
can also regard $\partial_{\alpha}\phi$ as a multi-arc that lies
on $\Sigma$ and connects $\mathbf x$ to $\mathbf y$. We define
$$A(\phi)=(\partial_{\alpha}\phi)\cdot\omega'.$$
Geometrically, if two Whitney disks $\phi_1,\phi_2$ differ by a
periodic domain $\mathcal P$, then
$$A(\phi_1)-A(\phi_2)=H(\mathcal P)\cdot\omega,$$
where $H(\mathcal P)\in H_2(Y)$ is the homology class represented
by $\mathcal P$.

 Let
$$\underline{\partial}\co \underline{CF}^{\infty}(Y,\omega;\mathcal L)\to
\underline{CF}^{\infty}(Y,\omega;\mathcal L)$$ be the boundary map
defined by
$$\underline{\partial}\:[\mathbf x,i]=\sum_{\mathbf y}\sum_{\stackrel{\scriptstyle\phi\in\pi_2(\mathbf x,\mathbf y)}
{\mu(\phi)=1}}\#\big(\mathcal M(\phi)/\mathbb R\big)
T^{A(\phi)}[\mathbf y,i-n_z(\phi)].$$ The chain homotopy type of
the chain complex
$$\big(\underline{CF}^{\infty}(Y,\omega;\mathcal
L),\underline{\partial}\big)$$ only depends on the homology class
$[\omega]\in H_1(Y)$. When $\omega$ is null-homologous in $Y$, the
coefficients are not ``twisted" at all, namely,
$$\underline{CF}^{\infty}(Y,\omega;\mathcal L)\cong CF^{\infty}(Y;\mathbb Q)\otimes_{\mathbb Q}\mathcal L.$$

The standard construction in Heegaard Floer homology
\cite{OSzAnn1} allows us to define the chain complexes
$\underline{\widehat{CF}}(Y,\omega;\mathcal L)$ and
$\underline{CF}^{\pm}(Y,\omega;\mathcal L)$. When $K$ is a
null-homologous knot in $Y$ and $\omega\subset Y-K$, we can define
the twisted knot Floer complex
$\underline{\widehat{CFK}}(Y,K,\omega;\mathcal L)$. The homologies
of the chain complexes are called twisted Heegaard Floer
homologies.

\subsection{Twisted chain maps}

Let
$(\Sigma,\mbox{\boldmath$\alpha$},\mbox{\boldmath$\beta$},\mbox{\boldmath$\gamma$},z)$
be a Heegaard triple-diagram. Let $\omega'\subset\Sigma$ be a
closed immersed curve which is in general position with the
$\alpha$--, $\beta$-- and $\gamma$--curves.

The pants construction in \cite[Subsection 8.1]{OSzAnn1} gives
rise to a four-manifold $X_{\alpha,\beta,\gamma}$ with
$$\partial
X_{\alpha,\beta,\gamma}=-Y_{\alpha,\beta}-Y_{\beta,\gamma}+Y_{\alpha,\gamma}\:.$$
By this construction $X_{\alpha,\beta,\gamma}$ contains a region
$\Sigma\times \triangle$, where $\triangle$ is a two-simplex with
edges $e_{\alpha},e_{\beta},e_{\gamma}$. Let $\omega'\times
[0,1]=\omega'\times e_{\alpha}\subset X_{\alpha,\beta,\gamma}$ be
the natural properly immersed annulus such that
$$\omega'\times\{0\}\subset
Y_{\alpha,\beta}\:,\quad\omega'\times\{1\}\subset
Y_{\alpha,\gamma}\:.$$

Suppose $\mathbf x\in\mathbb T_{\alpha}\cap\mathbb T_{\beta},
\mathbf y\in\mathbb T_{\beta}\cap\mathbb T_{\gamma},\mathbf
w\in\mathbb T_{\alpha}\cap\mathbb T_{\gamma}$, $\psi$ is a
topological Whitney triangle connecting them. Let
$\partial_{\alpha}\psi=\partial\psi\cap\mathbb T_{\alpha}$ be the
arc connecting $\mathbf x$ to $\mathbf w$. We can regard
$\partial_{\alpha}\psi$ as a multi-arc on $\Sigma$. Define
$$A_3(\psi)=(\partial_{\alpha}\psi)\cdot\omega'.$$

Let the chain map
$$\underline{f^{\infty}_{\alpha,\beta,\gamma,\:\omega'\times I}}\co \underline{CF}^{\infty}(Y_{\alpha,\beta},\omega'\times\{0\};\mathcal L)
\otimes_{\mathbb Q} CF^{\infty}(Y_{\beta,\gamma};\mathbb Q)\to
\underline{CF}^{\infty}(Y_{\alpha,\gamma},\omega'\times\{1\};\mathcal
L)$$ be defined by the formula:
$$\underline{f^{\infty}_{\alpha,\beta,\gamma,\:\omega'\times I}}([\mathbf x,i]\otimes[\mathbf y,j])=
\sum_{\mathbf w}\sum_{\stackrel{\scriptstyle\psi\in\pi_2(\mathbf
x,\mathbf y,\mathbf w )}{\mu(\psi)=0}}\#\mathcal
M(\psi)T^{A_3(\psi)}[\mathbf w,i+j-n_z(\psi)].$$

The standard constructions \cite{OSzAnn1,OSzAnn2} allow us to
define chain maps introduced by cobordisms. We also have the
surgery exact triangles. For example, suppose $K\subset Y$ is a
knot with frame $\lambda$, $\omega\subset Y-K$ is a closed curve,
then $\omega$ also lies in the manifolds $Y_{\lambda}$ and
$Y_{\lambda+\mu}$ obtained by surgeries. The $2$--handle addition
cobordism $W$ from $Y$ to $Y_{\lambda}$ naturally contains a
properly immersed annulus $\omega\times I$. We can define a chain
map induced by $W$:
$$\underline{f^{\infty}_{W,\:\omega\times I}}\co\underline{CF}^{\infty}(Y,\omega;\mathcal L)
\to\underline{CF}^{\infty}(Y_{\lambda},\omega;\mathcal L).$$
Similarly, there are two other chain maps induced by the
cobordisms $Y_{\lambda}\to Y_{\lambda+\mu}$ and
$Y_{\lambda+\mu}\to Y$. We then have the long exact sequence
\cite{AiP}:
\begin{equation}\label{Eq:TwistExact}
\begin{CD}
\cdots\to\underline{HF}^+(Y,\omega;\mathcal L)\to
\underline{HF}^+(Y_{\lambda},\omega;\mathcal L)\to
\underline{HF}^+(Y_{\lambda+\mu},\omega;\mathcal L)\to\cdots.
\end{CD}
\end{equation}

\subsection{Some properties of twisted Heegaard Floer homology}

Many properties of the untwisted Heegaard Floer homology have
analogues in the twisted case. For example, the connected sum
formula for $\underline{\widehat{HF}}$ is the following:
$$\underline{\widehat{HF}}(Y_1,\omega_1;\mathcal L)\otimes\underline{\widehat{HF}}(Y_2,\omega_2;\mathcal L)
\cong\underline{\widehat{HF}}(Y_1\#Y_2,\omega_1\cup\omega_2;\mathcal
L).$$

\begin{figure}
\begin{picture}(335,120)
\put(100,0){\scalebox{0.50}{\includegraphics*[130pt,350pt][470pt,
590pt]{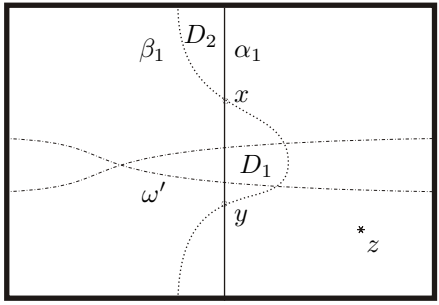}}}

\put(190,97){$\alpha_1$}

\put(154,97){$\beta_1$}

\put(190,80){$x$}

\put(190,35){$y$}

\put(240,23){$z$}

\put(192,53){$D_1$}

\put(171,102){$D_2$}

\put(155,42){$\omega'$}

\end{picture}
\caption{A Heegaard diagram for $S^2\times S^1$. The Heegaard
surface is the torus obtained by gluing the opposite sides of the
rectangle.}\label{HTorus}
\end{figure}

\begin{lem}\label{Lem:TrivialS2}
Suppose $Y$ contains a non-separating two-sphere $S$, $\omega\in
Y$ is a closed curve such that $\omega\cdot S\ne0$. We then have
$$\underline{\widehat{HF}}(Y,\omega;\mathcal L)=0,\quad\underline{HF}^+(Y,\omega;\mathcal L)=0.$$
\end{lem}
\begin{proof}
As in Figure~\ref{HTorus}, $\underline{\widehat{CF}}(S^2\times
S^1,\omega;\mathcal L)$ has two generators $x,y$. There are two
bigons $D_1,D_2$ connecting $x$ to $y$. If $\omega\cdot
(S^2\times\mathrm{point})=d\ne0$, we can assume $A(D_1)=d,
A(D_2)=0$, then $\underline{\partial} x=\pm(T^d-1)y$. This implies
that $\underline{\widehat{HF}}(S^1\times S^2,\omega;\mathcal L)=0$
since $T^d-1$ is invertible in $\mathcal L$.

If $Y$ contains a non-separating two-sphere $S$ and $\omega\cdot
S\ne0$, then $Y$ has a summand $S^2\times S^1$ such that
$\omega\cdot (S^2\times\mathrm{point})\ne0$. The connected sum
formula then shows that
$\underline{\widehat{HF}}(Y,\omega;\mathcal L)=0$.

For $\underline{HF}^+$, it follows from the long exact sequence
$$\begin{CD}
\cdots@>>>\underline{HF}^+@>U>> \underline{HF}^+@>>>
\underline{\widehat{HF}}@>>>\cdots
\end{CD}
$$
that $U$ is an isomorphism. For any element $x\in
\underline{HF}^+(Y,\omega;\mathcal L)$, $U^nx=0$ when $n$ is
sufficiently large, so $\underline{HF}^+(Y,\omega;\mathcal L)=0$.
\end{proof}

The following theorem is a twisted version of
\cite[Theorem~1.1]{Ni2}.

\begin{thm}\label{TwistKnotFibre}
Suppose $K$ is a null-homologous knot in a closed, oriented,
connected 3--manifold $Y$, $Y-K$ is irreducible, and $F$ is a
genus $g$ Seifert surface of $K$. Let $\omega\subset Y-K$ be a
closed curve. If
$\underline{\widehat{HFK}}(Y,K,\omega,[F],g;\mathcal L
)\cong\mathcal L$, then $K$ is fibred, and $F$ is a fibre of the
fibration.
\end{thm}
\begin{proof}
We could prove this theorem by repeating the whole proof in
\cite{Ni2}, but we would rather choose to apply
\cite[Theorem~1.1]{Ni2} directly.

Let $(M,\gamma)$ be the sutured manifold (see
\cite[Definition~2.1]{Ni2}) obtained by cutting $Y-K$ open along
$F$. The proof of \cite[Proposition~3.1]{Ni2} shows that $M$ is a
homology product. Hence we can glue $R_+(\gamma)$ to $R_-(\gamma)$
by a suitable homeomorphism to get a manifold with torus boundary,
which is the exterior of a knot $K'$ in a manifold $Y'$ with
$b_1(Y')=0$. This cut-and-reglue process can be realized by Dehn
surgeries on knots in $F$, so $\omega$ can be regarded as a curve
in $Y'-K'$. As in \cite[Proposition 3.5, the second proof]{Ni1},
using a filtered version of the exact sequence
(\ref{Eq:TwistExact}) and the adjunction inequality, we can show
that $\underline{\widehat{HFK}}(Y',K',\omega,[F],g;\mathcal L
)\cong\mathcal L$.

Since $b_1(Y')=0$, there is no real ``twisting" at all, namely,
$$\underline{\widehat{CFK}}(Y',K',\omega;\mathcal L
)\cong\widehat{CFK}(Y',K';\mathbb Q)\otimes_{\mathbb Q}\mathcal
L.$$ So we get ${\widehat{HFK}}(Y',K',[F],g;\mathbb Q)\cong\mathbb
Q$. Now \cite[Theorem~1.1]{Ni2} implies that $K'$ is fibred with
fibre $F$, hence so is $K$.
\end{proof}

\section{A homological version of the main theorem}

The goal of this section is to prove the following homological
version of the main theorem.

\begin{prop}\label{HomoProd}
Suppose $Y$ is a closed $3$--manifold, $F\subset Y$ is a closed
non-separating connected surface of genus $g\ge2$. Let $M$ be the
$3$--manifold obtained by cutting $Y$ open along $F$. The two
boundary components of $\partial M$ are denoted by $F_-,F_+$. If
$HF^+(Y,[F],g-1)\cong\mathbb Z$, then $M$ is a homology product,
namely, the two maps
$$i_{\pm*}\co H_*(F_{\pm})\to H_*(M)$$
are isomorphisms.
\end{prop}

We will make use of the fact that the Euler characteristic of
$HF^+$ is Turaev's torsion function. A Heegaard diagram for $Y$
will be constructed, and will be used to study the torsion
function of $Y$. Then Proposition~\ref{HomoProd} can be proved
using the same argument as \cite[Proposition 3.1]{Ni2}.

\subsection{A Heegaard diagram for $Y$}

\begin{construction}\label{RelMorse0}
We will construct a Heegaard diagram for $Y$ in a similar manner
as in \cite[Construction 2.10]{Ni2}.

\noindent{\bf Step 0.}\quad{\sl A relative Morse function}

Let $M$ be the compact manifold obtained by cutting open $Y$ along
$F$; the two components of $M$ are denoted by $F_-,F_+$. Let
$\psi\co F_+\to F_-$ be the gluing map. Consider a self-indexed
relative Morse function $u$ on $M$. Namely, $u$ satisfies:

(1) $u(M)=[0,3]$, $u^{-1}(0)=F_-$, $u^{-1}(3)=F_+$.

(2) $u$ has no degenerate critical points,
$u^{-1}\{\textrm{critical points of index $i$}\}=i$.

(3) $u$ has no critical points on $F_{\pm}$.

Let $F_{\#}=u^{-1}(\frac32)$.

Suppose $u$ has $r$ index $1$ critical points. Then the genus of
$F_{\#}$ is $g+r$. The gradient $-\nabla u$ generates a flow
$\phi_t$ on $M$. There are $2r$ points on $F_+$, which are
connected to index $2$ critical points by flowlines. We call these
points ``bad" points. Similarly, there are $2r$ bad points on
$F_-$, which are connected to index $1$ critical points by
flowlines.

\noindent{\bf Step 1.}\quad{\sl Construct some curves}

Choose two disjoint disks $D_+^a,D_+^b\subset F_+$. Flow the two
disks
 by $\phi_t$, their images on $F_{\#}$
and $F_-$ are $D^a_{\#},D^a_-,D^b_{\#},D^b_-$. (We choose the
disks generically, so that the flowlines starting from them do not
terminate at critical points.) We can suppose the gluing map
$\psi$ is equal to the intersection of the flow $\phi_t$ with
$F_-$ when restricted to $D^a_+\cup D^b_+$. Let
$A_{\pm}=F_{\pm}-\mathrm{int}(D^a_{\pm}\cup D^b_{\pm})$,
$A_{\#}=F_{\#}-\mathrm{int}(D^a_{\#}\cup D^b_{\#})$.

On $F_{\#}$, there are $r$ simple closed curves
$\alpha_{2g+1},\dots,\alpha_{2g+r}$, which are connected to index
$1$ critical points by flowlines. And there are $r$ simple closed
curves $\beta_{2g+1},\dots,\beta_{2g+r}$, which are connected to
index $2$ critical points by flowlines.

Choose $2g$ disjoint arcs $\xi^-_1,\dots,\xi^-_{2g}\subset A_-$,
such that their endpoints lie on $\partial D^b_-$, and they are
linearly independent in $H_1(A_-,\partial A_-)$. We also suppose
they are disjoint from the bad points. Let
$\xi^+_i=\psi^{-1}(\xi^-_i)$. We also flow back
$\xi^-_1,\dots,\xi^-_{2g}$ by $\phi_{-t}$ to $F_{\#}$, the images
are denoted by $\xi_1^{\#},\dots,\xi_{2g}^{\#}$.

Choose $2g$ disjoint arcs $\eta^+_1,\dots,\eta^+_{2g}\subset A_+$,
such that their endpoints lie on $\partial D^a_+$, and they are
linearly independent in $H_1(A_+,\partial A_+)$. We also suppose
they are disjoint from the bad points. Flow them by $\phi_t$ to
$F_{\#}$, the images are denoted by
$\eta_1^{\#},\dots,\eta_{2g}^{\#}$.

\noindent{\bf Step 2.}\quad{\sl Construct a diagram}

Suppose $[c_1,c_2]$ is a subinterval of $[0,3]$, let $(\partial
D^a_+)\times[c_1,c_2]$ be the annulus which is the image of
$\partial D^a_+$ inside $u^{-1}([c_1,c_2])$ under the flow
$\phi_t$. Similarly, define $(\partial D^b_+)\times[c_1,c_2]$. Let
$$\Sigma=A_+\cup A_{\#}\cup\big((\partial D^a_+)\times[\frac32,3]\big)\cup\big((\partial D^b_+)\times[0,\frac32]\big).$$
 Let
$$\alpha_i=\xi^+_i\cup\xi_i^{\#}\cup\{\textrm{2 arcs}\},$$
where the 2 arcs are vertical arcs connecting $\xi^+_i$ to
$\xi_i^{\#}$ on a corresponding annulus, $i=1,\dots,2g$.
Similarly, let
$$\beta_i=\eta^+_i\cup{\eta}_i^{\#}\cup\{\textrm{2 arcs}\}.$$
Let $\alpha_0=\partial{D^a_{\#}}$, $\beta_0=\partial{D^b_{\#}}$.

Let
\begin{eqnarray*}
\mbox{\boldmath${\alpha}$} &= &
\{\alpha_1,\dots,\alpha_{2g}\}\cup\{{\alpha}_{2g+1},\dots,{\alpha}_{2g+r}\}\cup\{\alpha_0\},\\
\mbox{\boldmath${\beta}$} &= &
\{\beta_1,\dots,\beta_{2g}\}\cup\{{\beta}_{2g+1},\dots,{\beta}_{2g+r}\}\cup\{\beta_0\}.
\end{eqnarray*}

\noindent{\bf Step 3.}\quad{\sl Check that
$(\Sigma,\mbox{\boldmath${\alpha}$}, \mbox{\boldmath$\beta$})$ is
a Heegaard diagram for $Y$}

This step is quite routine, we leave the reader to check the
following

\noindent(A) $\Sigma$ separates $Y$ into two genus $(2g+1+r)$
handlebodies $U_1,U_2$. Every curve in
$\mbox{\boldmath${\alpha}$}$ bounds a disk in $U_1$, every curve
in $\mbox{\boldmath${\beta}$}$ bounds a disk in $U_2$.

\noindent(B) $\Sigma-\mbox{\boldmath${\alpha}$}$ is connected,
$\Sigma-\mbox{\boldmath${\beta}$}$ is connected.

Then
$$(\Sigma,\mbox{\boldmath${\alpha}$},
\mbox{\boldmath$\beta$})$$ is a Heegaard diagram for $Y$. \qed
\end{construction}

\subsection{Preliminaries on Turaev's torsion function}

The Euler characteristic of $HF^+$ is equal to Turaev's torsion
function $T$. In this subsection we will briefly review the
definition of $T$. The readers are referred to \cite{Tu} if more
details are desired.

Suppose $Y$ is a closed oriented 3--manifold. The group
$H=H_1(Y;\mathbb Z)$ acts on $\mathrm{Spin}^c(Y)$. As in
\cite{OSzAnn1}, we denote this action by addition. Fix a finite
$CW$ decomposition of $Y$, $\widetilde Y$ is the maximal abelian
cover of $Y$ with its induced $CW$ structure. A family of cells
(of all dimensions) in $\widetilde Y$ is said to be {\it
fundamental} if over each cell of $Y$ lies exactly one cell of
this family. Choose a fundamental family of cells $\widetilde e$
in $\widetilde Y$, we get a basis for the cellular chain complex
$C_*(\widetilde Y)$ over the group ring $\mathbb Z[H]$. As shown
in \cite{Tu}, $\widetilde e$ also gives rise to a Spin$^c$
structure $\mathfrak s\in\mathrm{Spin}^c(Y)$.

Let $\mathbb F$ be a field and $\varphi\co H\to\mathbb F^{\times}$
be a group homomorphism, $\mathfrak s\in\mathrm{Spin}^c(Y)$, and
$\widetilde e$ be a fundamental family of cells which gives rise
to $\mathfrak s$. Then one can define $\tau^{\varphi}(Y,\mathfrak
s)\in\mathbb F$ to be the Reidemeister--Franz torsion of
$C_*(\widetilde Y)$ as in \cite[Section 2.3]{Tu}.

Let $Q(H)$ be the classical ring of quotients of the group ring
$\mathbb Q[H]$. $Q(H)$ splits as a finite direct sum of fields:
$$Q(H)=\oplus_{i=1}^n\mathbb F_i.$$ $\mathbb F_i$ is in the form
$\mathbb K_i(t_1,t_2,\dots,t_b)$, where $\mathbb K_i$ is a
cyclotomic field and $b=b_1(Y)$. Since $H\subset\mathbb
Q[H]\subset Q(H)$, there are natural projections $\varphi_i\co
H\to\mathbb F_i$. Turaev defined
$$\tau(Y,\mathfrak s)=\sum_{i=1}^n\tau^{\varphi_i}(Y,\mathfrak s)\in\oplus_{i=1}^n\mathbb F_i=Q(H),$$
and showed that $\tau(Y,\mathfrak s)\in\mathbb Z[H]$ when
$b_1(Y)\ge2$. The coefficients of $\tau(Y,\mathfrak s)$ gives the
torsion function $T$. More precisely, when $b_1(Y)\ge 2$, $T$ is
defined by the following formula:
\begin{equation}\label{TuraevTor}
\tau(Y,\mathfrak s)=\sum_{h\in H}T(Y,\mathfrak s-h)h;
\end{equation}
and when $b_1(Y)=1$, one can define
$T_t\co\mathrm{Spin}^c(Y)\to\mathbb Z$ in a similar way, once a
chamber $t$ of $H_1(Y;\mathbb R)$ is chosen.

Suppose the $CW$ decomposition of $Y$ has one $0$--cell, $m$
$1$--cells, $m$ $2$--cells and one $3$--cell. The chain complex
$C_*(\widetilde Y)=(C_3\to C_2\to C_1\to C_0)$, where
$C_0,C_1,C_2,C_3$ are free $\mathbb Z[H]$--modules with ranks
$1,m,m,1$, respectively. Let $h_1,\dots,h_m$ be the generators of
$H$ represented by the $1$--cells, and $g_1,\dots,g_m\in H$ be the
elements which are dual to the $2$--cells.

Denote by $\Delta_{r,s}$ the determinant of the matrix obtained
from the $m\times m$--matrix (over $\mathbb Z[H]$) of the boundary
homomorphism $C_2\to C_1$ by deleting the $r$th row and $s$th
column. Turaev proved the following equation (see (4.1.a) in
\cite{Tu}):
\begin{equation}\label{TorEq}
\tau(Y,\mathfrak s)(g_r-1)(h_s-1)=\pm\Delta_{r,s}\in\mathbb Z[H].
\end{equation}

When $\mathfrak s$ is a nontorsion Spin$^c$ structure, Ozsv\'ath
and Szab\'o showed in \cite{OSzAnn2} that
\begin{equation}\label{ChiHF}
\chi(HF^+(Y,\mathfrak s))=\pm T(Y,\mathfrak s).
\end{equation}

\subsection{Proof of the homological version}\label{PrHomoVer}

We deal with the case of $b_1(Y)\ge2$ first. We will use the
Heegaard splitting in Construction~\ref{RelMorse0} as the fixed
$CW$ decomposition of $Y$. Each $\alpha_i$ corresponds to a
$1$--handle in $Y$, let $a_i$ be the $1$--cell which is the core
of the $1$--handle; each $\beta_i$ corresponds to a $2$--handle in
$Y$, let $b_i$ be the $2$--cell which is the core of the
$2$--handle.

Let $\sigma\co H\to\mathbb Z$ be the group homomorphism given by
counting the intersection number with $[F]$. We construct the
universal abelian cover $\widetilde{Y}$ of $Y$ in two steps. First
we take the infinite cyclic cover of $Y$ dual to $\sigma$, denoted
by $Y^{\sigma}$, which is the union of infinitely many copies of
$M$:
$$Y^{\sigma}=\cdots\cup_{F_{-1}}M_{-1}\cup_{F_0}M_0\cup_{F_1}M_1\cup_{F_2}M_2\cup_{F_3}\cdots,$$
where $F_0,F_1$ are identified with $F_-,F_+$, respectively, if
$M_0$ is identified with $M$. Then we take the cover
$\pi\co\widetilde{Y}\to Y^{\sigma}$.

We choose a lift of the $0$--cell, a lift of the $3$--cell, and
$\widetilde{a_i},\widetilde{b_i}$ which are lifts of the
$1$--cells and $2$--cells. All lifts are chosen in
$\pi^{-1}(M_0)$. This fundamental family of cells gives rise to a
Spin$^c$ structure $\mathfrak s\in\mathrm{Spin}^c(Y)$.

We extend the group homomorphism $\sigma$ to a map of the group
rings
$$\sigma\co\mathbb Z[H]\to\mathbb Z[\mathbb Z]=\mathbb Z[q,q^{-1}].$$
By Equation (\ref{TuraevTor}), Equation (\ref{ChiHF}) and the
assumption that $$HF^+(Y,[F],g-1)\cong HF^+(Y,[F],1-g)\cong\mathbb
Z,$$ we conclude that $\sigma(\tau(Y,\mathfrak s))$ is a Laurent
polynomial of $q$ with degree $2g-2$, and the coefficients of the
highest term and the lowest term are $\pm1$. Here the degree of a
Laurent polynomial is defined to be the difference of the degree
of the highest term and the degree of the lowest term.

In Equation (\ref{TorEq}), we choose $g_r\in H$ to be the element
dual to the $2$--cell corresponding to $\beta_0$, and $h_s$ to be
the element represented by the $1$--cell corresponding to
$\alpha_0$. Note that if we cap off $\partial
A_+=\alpha_0\cup\beta_0$, we get the surface $F$ from $A_+$. We
thus have $\sigma(g_r)=q^{\pm1},\sigma(h_s)=q^{\pm1}$. So
$\sigma(\text{\rm LHS of (\ref{TorEq})})$ is a degree $2g$ Laurent
polynomial with leading coefficient $\pm1$. Here the leading
coefficient of a Laurent polynomial is defined to be the
coefficient of its lowest term.

Now we analyze the boundary map $\partial\co C_2\to C_1$. Suppose
$1\le i,j\le2g+r$, and the coefficient of $\widetilde{a_i}$ in
$\partial\widetilde{b_j}$ is $c_{ij}\in\mathbb Z[H]$. If $x\in
\alpha_i\cap \beta_j$ lies in $A_{\#}\subset F_{\#}$, then $x$
contributes a lift of $a_i$
 in $\pi^{-1}(M_0)$ to $\partial \widetilde{b_j}$ since $\widetilde{b_j}\subset \pi^{-1}(M_0)$;
 if $x\in \alpha_i\cap \beta_j$ lies in $A_{+}\subset F_{+}$,
then $x$ contributes a lift of $a_i$ in $\pi^{-1}(M_1)$ to
$\partial \widetilde{b_j}$. Let $c^{\#}_{ij}$ (or $c^+_{ij}$) be
the intersection number of $\alpha_i$ and $\beta_j$ inside the
domain $A_{\#}$ (or $A_+$), then we conclude that
$$\sigma(c_{ij})=c^{\#}_{ij}+c^+_{ij}q.$$
 If $i$ or $j$ is bigger than $2g$, then
$c^+_{ij}=0$.

Consider the matrix $\tensor{C}=(c_{ij})_{1\le i,j\le2g+r}$. The
result in the last paragraph implies that
$\sigma(\det(\tensor{C}))$ is a polynomial of degree at most $2g$,
and its constant term is $\det(c^{\#}_{ij})_{1\le i,j\le2g+r}$. By
(\ref{TorEq}), $\sigma(\det(\tensor{C}))$ is a degree $2g$ Laurent
polynomial with leading coefficient $\pm1$. Hence
\begin{equation}\label{Eq:Det=1}
\det(c^{\#}_{ij})_{1\le i,j\le2g+r}=\pm1.
\end{equation}

In Construction~\ref{RelMorse0}, let
$$N=M-\big(\mathrm{int}(D^+_b)\times[0,3]\big),\quad\gamma=(\partial D^+_b)\times[0,3].$$
$(N,\gamma)$ is a sutured manifold. We claim that $(N,\gamma)$ is
a homology product, namely, $$H_*(N,R_-(\gamma))\cong
H_*(N,R_+(\gamma))\cong0.$$ The proof is the same as
\cite[Proposition 3.1]{Ni2}. In fact, as in \cite[Lemmas 3.2 and
3.3]{Ni2}, using (\ref{Eq:Det=1}), one can show that
$H_2(N;\mathbb F)=0$ and the map
$$i_*\co H_1(R_-(\gamma),\partial R_-(\gamma);\mathbb F)\to H_1(N,\gamma;\mathbb F)$$
is injective for any field $\mathbb F$. Then the homological
argument as in \cite[Proposition 3.1]{Ni2} shows that $(N,\gamma)$
is a homology product. Since $M$ is obtained by capping off
$\gamma$ by $D^2\times I$, $M$ is also a homology product.

For the case of $b_1(Y)=1$, the proof is essentially the same.

\section{Characteristic product pairs}\label{CharPPair}

Using the surgery exact sequence and the adjunction inequality,
one can prove the following result. Details of the proof can be
found in \cite[Lemma 5.4]{OSzSympl} and \cite[Proposition 3.5, the
second proof]{Ni1}.

\begin{lem}[Ozsv\'ath--Szab\'o]\label{CutReglue}
Suppose $F$ is a closed connected surface in a closed manifold
$Y$, and the genus of $F$ is $g\ge2$. Let $Y'$ be the manifold
obtained by cutting open $Y$ along $F$ and regluing by a
self-homeomorphism of $F$. Let $HF^{\circ}$ denote one of
$\widehat{HF}$ and $HF^+$. Then we have
$$HF^{\circ}(Y,[F],g-1)\cong HF^{\circ}(Y',[F],g-1).$$\qed
\end{lem}

We also need the following two simple lemmas.

\begin{lem}\label{Lem:SplitProd}
Suppose $M$ is a compact $3$--manifold with two boundary
components $F_-,F_+$, and $M$ is a homology product, namely,
$$H_*(M,F_-)\cong H_*(M,F_+)\cong0.$$
Suppose $F_0\subset M$ is a closed surface that is homeomorphic to
$F_-$, and $F_0$ splits $M$ into two parts $M_-,M_+$,  $\partial
M_{\pm}=F_0\cup F_{\pm}$. Then both $M_-$ and $M_+$ are homology
products.
\end{lem}
\begin{proof}
From the exact sequence
$$\cdots\to H_1(M,F_-)\to H_1(M,M_-)\to H_0(M_-,F_-)\to\cdots$$
we conclude that $H_1(M,M_-)=0$, thus the inclusion map
$H_1(M_-)\to H_1(M)$ is surjective. Similarly, $H_1(M_+)\to
H_1(M)$ is surjective.

Let $g$ be the genus of $F_-$. Since $M$ is a homology product,
 $H_2(M)\cong \mathbb Z$ is generated
by $[F_-]=[F_+]$. So the maps $H_2(M_{\pm})\to H_2(M)$ are
surjective. We then have the short exact sequence
$$
0\to H_1(F_0)\to H_1(M_-)\oplus H_1(M_+)\to H_1(M)\to0.
$$
Both $H_1(M_-)$ and $H_1(M_+)$ surject onto $H_1(M)\cong
H_1(F_-)\cong \mathbb Z^{2g}\cong H_1(F_0)$, so $H_1(M_-)\cong
H_1(M_+)\cong\mathbb Z^{2g}$. It follows that the surjective maps
$H_1(M_{\pm})\to H_1(M)$ are actually isomorphisms. So we have the
exact sequences
$$
0\to H_2(M_{\pm})\to H_2(M)\to H_2(M,M_{\pm})\to0.
$$
We already know that the maps $H_2(M_{\pm})\to H_2(M)$ are
surjective, hence $H_2(M,M_{\pm})\cong0$.

Now we have $H_*(M_-,F_0)\cong H_*(M,M_+)\cong0$. The equality
$H_*(M_-,F_-)\cong0$ then follows from Poincar\'e duality and the
Universal Coefficients Theorem. Hence $M_-$ is a homology product,
and so is $M_+$ by the same argument.
\end{proof}

The next lemma is well-known, proofs of it can be found in
\cite{Mey,Sch}.

\begin{lem}\label{Lem:SCC}
A homology class on a closed, orientable surface is represented by
a simple closed curve if and only if it is primitive.\qed
\end{lem}

\begin{defn}
Suppose $S$ is a compact surface. The {\it norm} of $S$ is defined
by the formula:
$$x(S)=\sum_{S_i} \max\{0,-\chi(S_i)\},$$
where $S_i$ runs over all the components of $S$.
\end{defn}

The next theorem is an analogue of \cite[Theorem~6.2$'$]{NiCorr}.
Here we just sketch the proof, and refer the readers to
\cite[Section 6]{Ni2} and \cite{NiCorr} for more details.

\begin{thm}\label{rank>1}
Suppose $Y$ is a closed irreducible $3$--manifold, $F\subset Y$ is
a closed connected surface of genus $g\ge2$. Suppose
$\{F_1=F,F_2,\dots,F_n\}$ is a maximal collection of mutually
disjoint, nonparallel, genus $g$ surfaces in the homology class of
[F]. Cut open $Y$ along $F_1,\dots,F_n$, we get $n$ compact
manifolds $M_1,\dots,M_n$, $\partial M_k=F_k\cup F_{k+1}$, where
$F_{n+1}=F_1$. Let $\mathcal E_k$ be the subgroup of $H_1(M_k)$
spanned by the first homologies of the product annuli in $M_k$.

If $HF^+(Y,[F],g-1)\cong\mathbb Z$, then for each $k$, $\mathcal
E_k=H_1(M_k)$.
\end{thm}
\begin{proof}[Sketch of proof]
By Proposition~\ref{HomoProd} $M$ is a homology product.
Lemma~\ref{Lem:SplitProd} implies that each $M_k$ is also a
homology product.

Assume that $\mathcal E_k\ne H_1(M_k)$. By Lemma~\ref{Lem:SCC}, we
can find a simple closed curve $\omega\subset F_k$, such that
$[\omega]\notin\mathcal E_k$. Let $\omega_-=\omega\subset F_k$.
Since $M_k$ is a homology product, by Lemma~\ref{Lem:SCC} there is
a simple closed curve $\omega_+\subset F_{k+1}$ which is
homologous to $\omega$ in $M_k$. We fix an arc $\sigma\subset M_k$
connecting $F_k$ to $F_{k+1}$. Let $\mathcal S_m(+\omega)$ be the
set of properly embedded surfaces $S\subset M_k$, such that
$\partial S=\omega_-\sqcup(-\omega_+)$, and the algebraic
intersection number of $S$ with $\sigma$ is $m$. Here $-\omega_+$
denotes the curve $\omega_+$, but with opposite orientation.
Similarly, let $\mathcal S_m(-\omega)$ be the set of properly
embedded surfaces $S\subset M_k$, such that $\partial
S=(-\omega_-)\sqcup\omega_+$, and the algebraic intersection
number of $S$ with $\sigma$ is $m$. Since $M_k$ is a homology
product, $\mathcal S_m(\pm\omega)\ne\emptyset$. Let $x(\mathcal
S_m(\pm\omega))$ be the minimal value of $x(S)$ for all
$S\in\mathcal S_m(\pm\omega)$, where $x(S)$ is the norm of $S$.

\noindent{\bf Claim.} For positive integers $p,q$,
\begin{equation}\label{KeyFact} x(\mathcal S_p(+\omega))+x(\mathcal
S_q(-\omega))>(p+q)x(F_k). \end{equation}
\begin{proof}[Proof of
Claim.] Suppose $S_1\in\mathcal S_p(+\omega),S_2\in\mathcal
S_q(-\omega)$. Isotope $S_1,S_2$ so that they are transverse, then
perform oriented cut-and-paste to $S_1,S_2$, we get a closed
surface $P\subset \mathrm{int}(M_k)$, with $x(P)=x(S_1)+x(S_2)$.
Using standard arguments in 3--dimensional topology, we can assume
$P$ has no component which is a sphere or torus.

Since $M$ is a homology product, one can glue the two boundary
components of $M$ together to get a manifold $Z$, which has the
same homology as $F\times S^1$, so $H_2(Z)\cong H_1(F)\oplus
H_2(F)$. Thus if a closed surface $H\subset Z$ is disjoint from
one $F_k$, then $H$ must be homologous to a multiple of $F$.

$P$ is homologous to $(p+q)F_k$ in $Z$; in fact, as shown in
\cite{Ni2}, $P$ is the disjoint union of $p+q$ surfaces
$P_1,\dots,P_{p+q}$, where each $P_i$ is homologous to $F_k$.
Since $HF^+(Z,[F],g-1)\ne0$, $F$ is Thurston norm minimizing in
$Z$. So we have $x(P_i)\ge x(F)=x(F_k)$. So if $x(P)\le
(p+q)x(F_k)$, then the equality holds, and each $P_i$ has
$x(P_i)=x(F_k)$.

Next we claim that $P_i$ has only one component. Otherwise,
suppose $P_i=Q_1\sqcup Q_2$, then
$$x(Q_1),x(Q_2)<x(P_i)=x(F_k).$$ $[Q_1],[Q_2]$ are multiples of
$[F_k]$ in $H_2(Z)$. Since $F_k$ is Thurston norm minimizing in
$Z$ and $[F_k]$ is primitive, we must have $[Q_1]=[Q_2]=0$, which
is impossible.

Since $\{F_1,\dots,F_n\}$ is a maximal collection of disjoint,
nonparallel, genus $g$ surfaces, each $P_i$ is parallel to either
$F_k$ or $F_{k+1}$. Thus there exists $r\in\{0,1,\dots,p+q\}$,
such that $P_1,\dots,P_r$ are parallel to $F_k$,
$P_{r+1},\dots,P_{p+q}$ are parallel to $F_{k+1}$. Let
$C_r=P_r\cap S_1$. Then $C_r\times I$ is a collection of annuli
which connect $P_r$ to $P_{r+1}$, while $C_r$ is homologous to
$\omega$. This contradicts the assumption that
$[\omega]\notin\mathcal E_k$. Now the proof of the claim is
finished.\end{proof}

As shown in \cite[Lemma 6.4]{Ni2}, when $m$ is sufficiently large,
there exist connected surfaces $S_1\in\mathcal S_m(+\omega)$ and
$S_2\in\mathcal S_m(-\omega)$, such that they give taut
decompositions of $M_k$. By the work of Gabai \cite{G1}, one can
construct two taut smooth foliations $\mathscr F'_1,\mathscr F'_2$
of $M_k$, such that $F_k,F_{k+1}$ are leaves of the two
foliations; one can also construct a taut smooth foliation
$\mathscr F$ of $Z-\mathrm{int}(M_k)$ with compact leaves
$F_k,F_{k+1}$. Let $\mathscr F_i=\mathscr F'_i\cup\mathscr F$ be a
foliation of $Z$. Let $R$ be a connected surface in $Z$, whose
intersection with $F_k$ is $\omega$. As in \cite{Gh} or
\cite{Ni2}, using (\ref{KeyFact}), one can prove that $$\langle
c_1(\mathscr F_1),[R]\rangle\ne\langle c_1(\mathscr F_2),
[R]\rangle.$$ Thus \cite[Theorem 3.8]{Gh} can be applied to show
that $\mathrm{rank}(HF^+(Y,[F],g-1))>1$, a contradiction. This
finishes the proof of Theorem~\ref{rank>1}.
\end{proof}

\begin{cor}\label{CharSurj}
Suppose $(\Pi_k,\Psi_k)$ is the characteristic product pair
\cite[Definition 6]{NiCorr} for $M_k$, then the map
$$i_*\co H_1(\Pi_k)\to H_1(M_k)$$
is surjective.
\end{cor}
\begin{proof}
The proof is the same as \cite[Corollary 7]{NiCorr}.
\end{proof}

\begin{lem}\label{PuncTori}
Notation is as above, then each $\Pi_k$ contains a product
manifold $G_k\times I$, where $G_k$ is a once-punctured torus.
\end{lem}
\begin{proof} This is an easy consequence of Corollary~\ref{CharSurj}.  In
fact, let $E_k=\Pi_k\cap F_k$, $E^c_k=\overline{F_k-E_k}$, we can
construct a graph $\Gamma$ as follows. The vertices of $\Gamma$
correspond to the components of $E_k$ and $E^c_k$. $E_k\cap E^c_k$
consists of simple closed curves. For each of such curves, we draw
an edge connecting the two components of $E_k$ and $E^c_k$ that
are adjacent along the curve. No component of $E_k\cap E^c_k$ is a
non-separating curve in $F_k$, otherwise there would be a closed
curve $c\subset F_k$ which intersects the component exactly once,
thus $[c]\notin H_1(\Pi_k)$. It then follows that the $\Gamma$
contains no loop, so $\Gamma$ is a tree.

Consider a root of the tree, it corresponds to a component $H$ of
$E_k$ or $E^c_k$. $H$ has only one boundary component since it
corresponds to a root. $H$ is not a disk, so it contains a
once-punctured torus $G_k$. Since $H_1(G_k)$ contributes to
$H_1(M_k)$ nontrivially, $H$ must be a component of $E_k$.
\end{proof}

\section{Proof of the main theorem}

In this section, we will use Heegaard Floer homology with twisted
coefficients to prove Theorem~\ref{ClosedFibre}.

\begin{lem}\label{Lem:0surgery}
Suppose $Z$ is a closed $3$--manifold containing a non-separating
two-sphere $S$, $K\subset Z$ is a null-homologous knot, $H$ is a
genus $g$($>0$) Seifert surface for $K$. Let $Z_0(K)$ be the
manifold obtained by doing $0$--surgery on $K$, $\widehat{H}$ be
the extension of $H$ in $Z_0(K)$. Let $\omega\subset Z-K$ be a
closed curve such that $\omega\cdot S\ne0$. We then have
$$\underline{\widehat{HFK}}(Z,K,\omega,[H],g;\mathcal L)\cong \underline{HF}^+(Z_0(K),\omega,[\widehat H],g-1;\mathcal L).$$
\end{lem}
\begin{proof}
As in
 \cite[Corollary 4.5]{OSzKnot}, when $p$ is sufficiently large, we
have two exact triangles: (we suppress $[H],[\widehat H]$ and
$\mathcal L$ in the notation)
$$\begin{CD}
\cdots\to\underline{\widehat{HFK}}(Z,K,\omega,g)@>\sigma>>
\underline{HF}^+(Z_p,\omega,[g-1]) \to
\underline{HF}^+(Z,\omega)\to\cdots,
\end{CD}
$$
$$\begin{CD}
\cdots\to \underline{HF}^+(Z_0,\omega,[g-1])@>\sigma'>>
\underline{HF}^+(Z_p,\omega,[g-1]) \to
\underline{HF}^+(Z,\omega)\to\cdots.
\end{CD}
$$
By Lemma~\ref{Lem:TrivialS2}, $\underline{HF}^+(Z,\omega)=0$, so
the maps $\sigma,\sigma'$ are isomorphisms, hence our desired
result holds.
\end{proof}

\begin{proof}[Proof of Theorem~\ref{ClosedFibre}]
Notation is as in Section~\ref{CharPPair}. By
Lemma~\ref{PuncTori}, we have the product manifolds $G_k\times
I\subset M_k$. By cut-and-reglue along $F_k$'s, we can get a new
manifold $Y_1$ such that the $G_k\times I$'s are matched together
to form an essential submanifold $G\times S^1$ in $Y_1$, where $G$
is a once-punctured torus in $F$.

Since each $M_k$ is a homology product, we can construct a new
manifold $Y_2$ with $b_1(Y_2)=1$ by cutting $Y_1$ open along $F$
and then regluing by a homeomorphism of $F$.

By Lemma~\ref{CutReglue}, $$HF^+(Y_1,[F],g-1)\cong
HF^+(Y_2,[F],g-1)\cong HF^+(Y,[F],g-1)\cong\mathbb Z.$$

Let $D\subset G$ be a small disk. We remove $D\times S^1$ from
$Y_1$, then glue in a solid torus $V$, such that the meridian of
$V$ is $p\times S^1$ for a point $p\in\partial D$. The new
manifold is denoted by $Z$, and the core of $V$ is a
null-homologous knot $K$ in $Z$. $\check F=F-\mathrm{int}(D)$ is a
Seifert surface for $K$. Conversely, $Y_1$ can be obtained from
$Z$ by $0$--surgery on $K$.

$Z$ contains non-separating spheres. In fact, pick any properly
embedded non-separating arc $c\subset G-\mathrm{int}(D)$, such
that $\partial c\subset\partial D$, then $\partial(c\times S^1)$
bounds two disks in $V$. The union of the two disks and $c\times
S^1$ is a non-separating sphere in $Z$. Suppose $S$ is such a
non-separating sphere. Let $\omega\subset Z-\check F$ be a closed
curve such that $\omega\cdot S\ne0$. The curve $\omega$ can also
be viewed as lying in $Y_1$ and $Y_2$, and $\omega$ is disjoint
from $F$.

Since $b_1(Y_2)=1$ and $\omega\cdot[F]=0$, $\omega$ is
null-homologous in $Y_2$, so
$$\underline{HF}^+(Y_2,\omega,[F],g-1;\mathcal L)\cong HF^+(Y_2,[F],g-1;\mathbb Q)\otimes_{\mathbb Q}\mathcal L\cong\mathcal L.$$
A twisted version of Lemma~\ref{CutReglue} then implies that
$$\underline{HF}^+(Y_1,\omega,[F],g-1;\mathcal L)\cong\mathcal
L.$$

By Lemma~\ref{Lem:0surgery},
$\underline{\widehat{HFK}}(Z,K,\omega,[\check F],g;\mathcal
L)\cong\underline{HF}^+(Y_1,\omega,[F],g-1;\mathcal
L)\cong\mathcal L$. Now by Theorem~\ref{TwistKnotFibre} $K$ is
fibred with fibre $\check F$, hence $Y_1$ is fibred with fibre
$F$, and so is $Y$.
\end{proof}


\begin{thebibliography}{H}

\bibitem{AiP}{\bf Y. Ai, T. Peters}, {\it The twisted Floer homology of torus bundles}, preprint (2008), available at arXiv:0806.3487

\bibitem{G1}{\bf D. Gabai}, {\it Foliations and the topology of $3$--manifolds}, J. Differential Geom. 18 (1983), no. 3, 445--503

\bibitem{G3}{\bf D. Gabai}, {\it Foliations and the topology of $3$--manifolds III}, J. Differential Geom. 26 (1987), no. 3, 479--536


\bibitem{Gh}{\bf P. Ghiggini}, {\it Knot Floer homology detects genus-one fibred
knots},   Amer. J. Math.  130  (2008),  no. 5, 1151--1169

\bibitem{JM}{\bf S. Jabuka, T. Mark}, {\it Product formulae for Ozsv\'ath--Szab\'o $4$--manifold
invariants}, Geom. Topol. 12 (2008) 1557--1651 (electronic)

\bibitem{Ju}{\bf A. Juh\'asz}, {\it Floer homology and surface
decompositions}, Geom. Topol.  12  (2008), 299--350 (electronic)

\bibitem{Mey}{\bf M. Meyerson}, {\it Representing homology classes of closed orientable
surfaces}, Proc. Amer. Math. Soc.  61  (1976), no. 1, 181--182
(1977)

\bibitem{Ni1}{\bf Y. Ni}, {\it Sutured Heegaard diagrams for knots},
Algebr. Geom. Topol. 6 (2006), 513--537 (electronic)

\bibitem{Ni2}{\bf Y. Ni}, {\it Knot Floer homology detects fibred
knots}, Invent. Math. 170 (2007), no. 3, 577--608

\bibitem{NiCorr}{\bf Y. Ni}, {\it Corrigendum to ``Knot Floer homology detects
fibred knots"}, available at arXiv:0808.0940

\bibitem{OSzAnn1}{\bf P. Ozsv\'ath, Z. Szab\'o}, {\it Holomorphic disks and topological invariants for closed
three-manifolds}, Ann. of Math. (2), 159 (2004), no. 3, 1027--1158

\bibitem{OSzAnn2}{\bf P. Ozsv\'ath, Z. Szab\'o}, {\it Holomorphic disks and three-manifold invariants: properties and
applications}, Ann. of Math. (2), 159 (2004), no. 3, 1159--1245

\bibitem{OSzKnot}{\bf P. Ozsv\'ath, Z. Szab\'o}, {\it Holomorphic disks and knot invariants},
Adv. Math. 186 (2004), no. 1, 58--116

\bibitem{OSzSympl}{\bf P. Ozsv\'ath, Z. Szab\'o}, {\it Holomorphic triangle invariants and the topology of symplectic
four-manifolds}, Duke Math. J. 121 (2004), no. 1, 1--34

\bibitem{OSzGenus}{\bf P. Ozsv\'ath, Z. Szab\'o}, {\it Holomorphic disks and genus bounds}, Geom. Topol. 8 (2004), 311--334 (electronic)

\bibitem{Ra2}{\bf J. Rasmussen}, {\it Floer homology and knot complements}, PhD Thesis, Harvard University (2003), available at arXiv:math.GT/0306378

\bibitem{SW}{\bf S. Sarkar, J. Wang}, {\it  An algorithm for computing some Heegaard Floer homologies}, preprint (2006), to appear in Ann. of Math.,
available at arXiv:math.GT/0607777

\bibitem{Sch}{\bf J. Schafer}, {\it Representing homology classes on surfaces},  Canad. Math. Bull.
19 (1976), no. 3, 373--374

\bibitem{Tu}{\bf V. Turaev}, {\it Torsion invariants of $\mathrm{Spin}^c$--structures on
$3$--manifolds},  Math. Res. Lett.  4  (1997),  no. 5, 679--695
\end{thebibliography}
\end{document}